\documentclass[12pt]{amsart}
\usepackage{amssymb,latexsym}
\usepackage{enumerate}

\makeatletter
\@namedef{subjclassname@2010}{%
  \textup{2010} Mathematics Subject Classification}
\makeatother

\newtheorem{proposition}{Proposition}
\newtheorem{theorem}[proposition]{Theorem}
\newtheorem{lemma}[proposition]{Lemma}

\theoremstyle{definition}
\newtheorem{definition}[proposition]{Definition}

\theoremstyle{definition}

\theoremstyle{definition}

\numberwithin{equation}{section}

\newcommand{\CC}{{\mathbb C}}

\newcommand{\RR}{{\mathbb R}}

\newcommand{\NN}{{\mathbb N}}

\renewcommand{\Im}{\mbox{Im}}

\renewcommand{\Re}{\mbox{Re}}

\renewcommand{\date}{\today}
\def \bar{\overline}
\def \hat{\widehat}

\begin{document}

\vskip 3mm

\title[Polynomial Interpolation and Approximation in $\CC^d$ ]{\bf Polynomial Interpolation and Approximation in $\CC^d$ }  
\author{T. Bloom*, L. P. Bos, J.-P. Calvi and N. Levenberg}
\thanks{*Supported in part by an NSERC of Canada grant}
\date

\address{University of Toronto, Toronto, Ontario M5S 2E4 Canada}  
\email{bloom@math.toronto.edu}

\address{University of Verona, Verona, Italy}  
\email{leonardpeter.bos@univr.it}

\address{Universit\'e de Toulouse (III), Toulouse, France}
\email{jean-paul@calvi.org.uk}

\address{Indiana University, Bloomington, IN 47405 USA}
\email{nlevenbe@indiana.edu}

\begin{abstract}
We update the state of the subject approximately 20 years after the publication of \cite{BBCL}. 
This report is mostly a survey, with a sprinkling of assorted new results throughout. \end{abstract}

\maketitle

\section{Introduction.} Let $z_0,...z_n$ be $n+1$
distinct points in the plane and let $f$ be a
function which is defined at these points. The polynomials
$l_j(z)=\prod_{k\not= j}(z-z_k)/\prod_{k \not= j}(z_j-z_k), \ j=0,...,n$, 
\ are polynomials of degree $n$ with $l_j(z_k)=\delta_{j,k}$ which we call the
{\it fundamental Lagrange interpolating polynomials}, or FLIP's,
associated to $z_0,...,z_n$.  The polynomial $p(z)=\sum_{j=0}^n
f(z_j)l_j(z)$ is then the unique polynomial of degree at most $n$
satisfying $p(z_j)=f(z_j), \ j=0,...,n;$ we call it the {\it Lagrange
interpolating polynomial,} or LIP,
associated to $f,z_0,...,z_n.$ If $\Gamma$ is a
rectifiable Jordan curve such that
the points $z_0,...,z_n$ are inside $\Gamma$, and 
$f$ is holomorphic inside and on $\Gamma$, we can estimate
the error in our approximation of $f$ by $p$ at points
inside $\Gamma$ using the Hermite Remainder Formula: for any $z$ inside $\Gamma$,
\begin{equation} \label{hrf}
  f(z)-p(z)
  = 
  {1 \over 2\pi i} \int_{\Gamma} {\omega(z) \over \omega(t)} 
  {f(t) \over (t-z)} \, dt, 
\end{equation}
where $\omega (z)=\prod _{k=0} ^n (z-z_k).$ 
This elementary yet fundamental formula is the key to proving many important results on polynomial approximation and interpolation. We first recall the following result of Walsh which gives a quantitative version of the classical Runge theorem.

\begin{theorem}\label{walsh} {\bf (Walsh)} Let $K$ be a 
compact subset of the plane such that
$\CC \backslash K$ is connected and
has a Green function $g_K$.  Let $R > 1$, and define 
$$D_R:=\{z\in \CC: g_K(z) < \log R\}.$$ 
For $f$ continuous on $K$, let 
\begin{equation} \label{dndef} d_n(f,K):=\inf\{||f-p_n||_K: p_n \ \hbox{polynomial of degree} \ \leq n\}\end{equation}
where $||f-p_n||_K=\sup_{z\in K}|f(z)-p_n(z)|$. Then
$$
  \limsup_{n\to \infty} d_n(f,K)^{1/n}\leq 1/R
$$ 
if and only if $f$ is the restriction to $K$ of a function
holomorphic in $D_R$.  \end{theorem}
 
Here, $\CC \backslash K$ has a Green function $g_K$ means that $g_K$ is continuous and subharmonic in $\CC$, harmonic in $\CC \backslash K$ with $g_K(z)-\log |z|$ bounded as $|z|\to \infty$, and $g_K=0$ on $K$. This final condition says that $K$ is a {\it regular} compact set.

Consider the following situation. 
Let $\{z_{nj}\}, \ j=0,...,n; \ n=1,2,...$ be an array of points. For each $f$ defined in a neighborhood of this array, we can form the sequence of LIP's $\{L_nf\}$ associated to $f$. Let $\omega_n (z):=\prod _{j=0} ^n (z-z_{nj})$. An easy consequence of Theorem \ref{walsh} and (\ref{hrf}) is the following.

\begin{theorem} \label{walsh2} Let $K\subset \CC$ be compact and regular with $\CC\setminus K$ connected. Let $\{z_{nj}\}$ be an array of points in $K$. Then for any $f$ which is holomorphic in a neighborhood of $K$, we have $L_nf\rightrightarrows f$ on $K$ (uniform convergence) if and only if 
$$ \lim_{n\to \infty} |\omega_n(z)|^{\frac{1}{n+1}} =\delta(K)\cdot e^{g_K(z)}$$
uniformly on compact subsets of $\CC\setminus K$.
\end{theorem}

Here $\delta(K)$ is the transfinite diameter of $K$. In \cite{BBCL}, several conditions on the array $\{z_{nj}\}$ were discussed which imply for any $f$ which is holomorphic in a neighborhood of $K$, we have $L_nf\rightrightarrows f$ on $K$. The results in this univariate setting are well understood. In $\CC^d, \ d>1$, knowledge of Lagrange interpolation is less complete. Let $\mathcal P_n$ denote the complex vector space of holomorphic polynomials of degree at most $n$ and let 
$$N =N(n) :=\hbox{dim} \mathcal P_n =   {n+d \choose n}.$$
Thus 
$$\mathcal P_n= \hbox{span} \{e_1,...,e_N\}$$ 
where $\{e_j(z):=z^{\alpha(j)}\}$ are the standard basis monomials. We let 
$$l_n:= \sum_{j=1}^N {\rm deg} e_j=\frac{dnN}{d+1}.$$
For 
points $\zeta_1,...,\zeta_N\in \CC^d$, define a (generalized) {\it Vandermonde determinant} of order $n$ as 
\begin{equation}\label{vdmordn}VDM(\zeta_1,...,\zeta_N)=\det [e_i(\zeta_j)]_{i,j=1,...,N} \end{equation}
$$= \det
\left[
\begin{array}{ccccc}
 e_1(\zeta_1) &e_1(\zeta_2) &\ldots  &e_1(\zeta_N)\\
  \vdots  & \vdots & \ddots  & \vdots \\
e_N(\zeta_1) &e_N(\zeta_2) &\ldots  &e_N(\zeta_N)
\end{array}
\right].$$
Given $N$ points
$A_n=\{A_{n1},...,A_{nN}\}$ with 
$$VDM(A_{n1},...,A_{nN})\not= 0,$$ we can form the FLIP's
$$
  l_{nj}(x):= {VDM(A_{n1},...,x,...,A_{nN}) \over VDM(A_{n1},...,A_{nN})}, 
  \qquad j=1,...,N.
$$
In the one (complex) variable case, we get {\it cancellation} in this ratio
so that the formulas for the FLIP's simplify.
In general, we still have $l_{nj}(A_{ni})= \delta_{ji}$ and $l_{nj}\in
{\mathcal P}_n$
since $l_{nj}$ is a linear combination of $e_1,..,e_{N}$.
For $f$ defined at the points in $A_n$, 
\begin{equation}\label{lipcn}
  (L_nf)(x):= \sum_{j=1}^{N} f(A_{nj})l_{nj}(x)
\end{equation}
is the Lagrange interpolating polynomial (LIP) for $f$ and the points in $A_n$.

In one variable, $VDM(A_{n1},...,A_{nN})\not= 0$ provided the points in $A_n$ are distinct. Given a compact set $K\subset \CC^d$, we say that $K$ is {\it determining} for $\bigcup {\mathcal P}_n$ 
if whenever $h\in \bigcup {\mathcal P}_n$ satisfies 
$h=0$ on $K$, it follows that $h\equiv 0$. For these sets we can find
points $\{A_{n1},...,A_{nN}\}$ for each $n$ with $VDM(A_{n1},...,A_{nN})\not= 0$; we call these points {\it unisolvent} of degree $n$. Despite the lack of a Hermite-type remainder formula, we can describe one condition on an array $\{A_{nj}\}_{j=1,..., N; \  n=1,2,...}$ lying in a compact set $K\subset \CC^d$ satisfying a multivariate version of ``regular with $\CC\setminus K$ connected'' which implies for any $f$ which is holomorphic in a neighborhood of $K$, we have $L_nf\rightrightarrows f$ on $K$. We call 
$$
  \Lambda_n:= \sup_{z\in K} \Lambda_n(z):= \sup_{z\in K} \sum_{j=1}^{N} |l_{nj}(z)|
$$
the $n$-th {\it Lebesgue constant} for $K,A_n$ (the function $z\to \Lambda_n(z)$ is the $n-$th {\it Lebesgue function}). It is the norm of the linear operator $\mathcal L_n: C(K)\to \mathcal P_n\subset C(K)$ where $\mathcal L_n(f)= L_nf$ from (\ref{lipcn}) and we equip $C(K)$ with the uniform norm. The next result follows from a multivariate version of Theorem \ref{walsh} together with the Lebesgue inequality which says that for every continuous function $f$  on $K$ we have
\begin{equation} \label{lebineq} ||f - L_n f||_K \leq (1+\Lambda_n)d_n(f,K)\end{equation}
with $d_n(f,K)$ as in (\ref{dndef}) using polynomials in $\CC^d$. 
\begin{proposition} \label{startprop}
Let $K\subset \CC^d$ be polynomially convex and $L-$regular and let
$A_n\subset K$ satisfy $VDM(A_{n1},...,A_{nN})\not= 0$ for each $n=1,2,..$. If $\limsup_{n\to \infty} \Lambda_n^{1/n}=1$, for each $f\in C(K)$, $$\limsup_{n\to \infty} ||f-L_nf||_K^{1/n}=
\limsup_{n\to \infty} d_n(f,K)^{1/n}.$$ \end{proposition}

\noindent The notions of polynomial convexity and $L-$regularity will be defined in the next section. 

This property $\limsup_{n\to \infty} \Lambda_n^{1/n}=1$ is one of several we consider in the definition below.

\begin{definition} \label{props} Let $K$ be compact. Consider the following four properties which an array $\{A_{nj}\}_{j=1,..., N; \  n=1,2,...}\subset K$ may or may not possess:
\begin{enumerate}
\item  $\lim_{n\to \infty} \Lambda_n^{1/n} =1$;
\item  $\lim_{n\to \infty}|VDM(A_{n1},...,A_{nN})|^{1/l_n}=\delta(K)$;
\item  $\lim_{n\to \infty} \frac{1}{N}\sum_{j=1}^{N} \delta_{A_{nj}}= \mu_K$ weak-*;
\item $L_nf\rightrightarrows f$ on $K$ for each $f$ holomorphic on a neighborhood of $K$.
\end{enumerate}
\end{definition}

Here $\delta_A$ denotes the unit point mass at $A$. The probability measure $\mu_K$ is the (pluri-)potential theoretic equilibrium measure of $K$; i.e., for $K \subset \CC$ nonpolar, $\mu_K=\frac{1}{2\pi}\Delta V_K^*$ and  for $K \subset \CC^d$ nonpluripolar with $d>1$, $\mu_K=\frac{1}{(2\pi)^d}(dd^c V_K^*)^d$, the complex Monge-Amp\`ere measure of $V_K^*(z):=\limsup_{\zeta \to z}V_K(\zeta)$ where 
$$V_K(z)= \sup \{{1\over {\rm deg} (p)}\log |p(z)|:p\in \bigcup {\mathcal P}_n, \  ||p||_{K}\leq 1 \}$$
$$=\sup\{u(z):u\in L(\CC^d) \ \hbox{and} \  u\leq 0 \ \hbox{on} \ K\}.$$
Here $L(\CC^d)$ is the set of all plurisubharmonic functions on $\CC^d$ of logarithmic growth; i.e., $u\in L(\CC^d)$ if $u$ is plurisubharmonic in $\CC^d$ and $u(z)=\log |z| +0(1)$ as $|z|\to \infty$. 
 In \cite{BBCL} it was shown that for $K\subset \CC$ regular; i.e., $V_K=V_K^*=g_K$, and $\CC\setminus K$ connected, we have the implications
 \begin{equation} \label{eq:basicimplications}(1) \implies (2)\implies (3) \implies (4), \end{equation}
 while none of the reverse implications are necessarily true (although for arrays lying on the boundary of $K$, (3) and (4) are equivalent; for a more precise discussion, see \cite{BC2}). Proposition \ref{startprop} being true in $\CC^d$ for any $d$ shows that the implication $(1)\implies (4)$ remains true in $\CC^d$; and, as was shown in \cite{BBCL}, $(1) \implies (2)$ as well. 
  
We continue in the next section with the necessary definitions and an elaboration on the relationship between conditions (1) and (2). Recent deep results of R. Berman and S. Boucksom (\cite{BBnew} and with Nystrom \cite{BBN}) yield that  $(2) \implies (3)$; we discuss consequences of this result on recovering the measure $\mu_K$ in section 3. In section 4 we describe methods of recovering the extremal function $V_K$. We discuss the important Bernstein-Markov property in section 5. A brief introduction to weighted pluripotential theory in $\CC^d$ is provided in section 6, and a connection with unweighted pluripotential theory in $\CC^{d+1}$ as in  \cite{bloomweight} is given. Section 7 provides explicit and semi-explicit constructions of arrays in certain compact sets  satisfying conditions related to (1)-(4). We give a reprise of the analysis of so-called Bos arrays on the real unit disk $B_2\subset \RR^2\subset \CC^2$ in section 8. In section 9, we discuss computational approaches to constructing arrays in a compact set $K$ satisfying some of the properties (1)-(4). A brief discussion of Kergin interpolation forms the content of section 10, and we conclude this work, as was done in \cite{BBCL}, with a list of ten open problems.

We would like to thank the organizers of the Conference on Several Complex Variables on the occasion of Professor J\'ozef Siciak's 80th birthday for their hospitality and we dedicate this work to Professor Siciak for his contributions and inspiration to the pluripotential theory community.

\section{Subexponential Lebesgue constants and asymptotic Fekete arrays.} 

We work in $\CC^d$ using the same notation as in section 1. For a compact subset $K\subset \CC^d$ let
$$V_n =V_n(K):=\max_{\zeta_1,...,\zeta_N\in K}|VDM(\zeta_1,...,\zeta_N)|.$$
Then 
\begin{equation}\label{zahlim} \delta(K)=\delta^1(K)=\lim_{n\to \infty}V_{n}^{\frac{d+1}{dnN}} \end{equation} 
is the {\it transfinite diameter} of $K$. Points $z_1,...,z_N\in K$ satisfying $V_n= |VDM(z_1,...,z_N)|$ are called {\it $n-$th order Fekete points for $K$}. The temporary superscript ``$1$'' refers to a weight $w\equiv 1$ (see section 6). Zaharjuta \cite{zah} showed that the limit exists. Clearly if a compact set $K$ is contained in an algebraic subvariety of $\CC^d$ then $\delta(K)=0$. It turns out that for $K\subset \CC^d$ compact, $\delta(K)=0$ if and only if $K$ is pluripolar \cite{LT}.

If the compact set $K \subset \CC^d$ is $L$-regular, meaning that $V_K=V_K^*$, 
and for $R > 1$ we define
\begin{equation}\label{ten} 
  D_R := \{ z : V_K (z) < \log R \};
\end{equation}
then we have the Bernstein-Walsh inequality 
$$  |p(z)| \leq ||p||_K R^{ {\rm deg} p },  \qquad   z \in D_R $$
for every polynomial $p$ in $\CC^d$. A compact set $K\subset \CC^d$ is 
{\it polynomially convex}
if $K$ coincides with its {\it polynomial hull}
$$
  \hat K:= \{z\in \CC^d:|p(z)|\leq ||p||_K,
 \ p \ {\rm polynomial}\}.
$$
Then Theorem \ref{walsh} of Walsh goes over
{\it exactly} to several complex variables:
 
\begin{theorem} [\cite{sic81}] \label{sicwalsh} Let $K$ be an $L$-regular, polynomially convex
compact set in $\CC^d$. 
Let $R > 1$, and let $D_R$ be defined by (\ref{ten}).
Let $f$ be continuous on $K$. Then
$$
  \limsup_{n\to \infty} d_n(f,K)^{1/n}\leq 1/R
$$ 
if and only if $f$ is
the restriction to $K$ of a function holomorphic in $D_R$.  \end{theorem}

The ``only if'' direction is the same for any $d$ and uses the Bernstein-Walsh inequality. Theorem  \ref{sicwalsh} immediately yields Proposition \ref{startprop}, showing that for $L$-regular, polynomially convex
compact sets in $\CC^d$, condition (1) implies (4). Unless otherwise noted, when discussing conditions (1)--(4) of Definition \ref{props} we will always assume $K$ is $L$-regular and polynomially convex.

It is easy to see that (1) implies (2) but the converse is not true. On pp. 462-463 in \cite{BBCL}, it was observed that for an array $\{A_{nj}\}\subset K$ with 
$$|VDM(A_{n1},...,A_{nN})|=c_n V_n(K)$$ 
where 
$$0<c_n <1, \  \limsup_{n\to \infty} c_n^{1/n}<1, \ \hbox{and} \ \lim_{n\to \infty} c_n^{1/l_n}=1$$
(e.g., $c_n = v^n$ for $0<v<1$), property (2) holds but (1) does not. More precisely, we have the following.

\begin{proposition} \label{onetwo} Let $\{A_{nj}\}_{j=1,..., N; \  n=1,2,...}\subset K$ be an array of points. Suppose that 
$$\lim_{n\to \infty} \bigl(\frac{V_n(K)}{|VDM(A_{n1},...,A_{nN})|}\bigr)^{1/n}=1.$$
Then $(1)$ holds.
\end{proposition}
\begin{proof} The result follows trivially from the observation that if 
$$\frac{V_n(K)}{|VDM(A_{n1},...,A_{nN})|}\leq a_n,$$
then $\Lambda_n\leq N\cdot a_n$. This observation is a consequence of the fact that each FLIP can be written as
$$
  l_{nj}(z):= {VDM(A_{n1},...,z,...,A_{nN}) \over VDM(A_{n1},...,A_{nN})}$$
  so that
  $$|l_{nj}(z)| \leq a_n{|VDM(A_{n1},...,z,...,A_{nN})| \over V_n(K)}.$$
 Since $|VDM(A_{n1},...,z,...,A_{nN})| \leq V_n(K)$ for each $z\in K$, we have $||l_{nj}||_K \leq a_n$.\end{proof}

\bigskip

\section{Arrays yielding $\mu_K$.} In one variable, $\frac{-1}{n^2}\log |VDM(z_1,...,z_n)|$ is a discrete ``approximation'' to the logarithmic energy of the measure $\mu_n=\frac{1}{n}\sum_{j=1}^n\delta_{z_j}$. This is the idea behind the classical proof that $(2)\implies (3)$. In several complex variables, the complex Monge-Amp\`ere operator is non-linear and, until recently, no reasonable notion of the energy of a measure existed. We state without proof the remarkable result of Berman, Boucksom and Nystrom \cite{BBN} that, nevertheless, $(2)\implies (3)$ for general nonpluripolar compact sets $K$ in the multivariate setting.

\begin{theorem}[\cite{BBN}] \label{asympfek}  Let $K\subset \CC^d$ be compact and nonpluripolar. For each $n$, take points $x_1^{(n)},x_2^{(n)},\cdots,x_N^{(n)}\in K$ for which 
\begin{equation}\label{afp}
 \lim_{n\to \infty}|VDM(x_1^{(n)},\cdots,x_N^{(n)})|^{{(d+1)\over  dnN}}=\delta(K)
\end{equation}
$($asymptotic Fekete points$)$ and let $\mu_n:= \frac{1}{N}\sum_{j=1}^N \delta_{x_j^{(n)}}$. Then
$$
\mu_n \to \mu_{K} \ \hbox{weak}-*.
$$
\end{theorem}

This gives a positive answer to question 6 posed in \cite{BBCL}. In Proposition 3.7 of \cite{BBCL} it was shown that for a {\it Leja sequence} $\{x_1,x_2,...\}\subset K$, 
$$\lim_{n\to \infty} |VDM(x_1,...,x_N)|^{{(d+1)\over  dnN}}=\delta(K).$$
Thus the asymptotic Fekete property (\ref{afp}) holds for this {\it sequence} of points; so from Theorem \ref{asympfek} it follows that the discrete measures $$\mu_n :={1\over N}\sum_{i=1}^N\delta_{x_i}$$
converge weak-* to $\mu_K$. A Leja sequence is defined inductively as follows. Take the standard monomial basis $\{e_1,e_2,...\}$ for $\bigcup \mathcal P_n$ ordered so that deg$e_i\leq$deg$e_j$ if $i\leq j$. Given $m$ points $z_1,...,z_m$ in $\CC^d$, we write
$$VDM(z_1,...,z_m)=\det [e_i(z_j)]_{i,j=1,...,m}.$$
Starting with any point $x_1\in K$, having chosen $x_1,...,x_m\in K$ we choose $x_{m+1}\in K$ so that
$$|VDM(x_1,...,x_m,x_{m+1})|=\max_{x\in K} |VDM(x_1,...,x_m,x)|.$$

We remark that despite possessing the desirable property that $\mu_n \to \mu_K$ weak-*, it is unknown if (1) always holds for a Leja sequence, even in the univariate case ($d=1$). This is the first question in section 5 of \cite{BBCL}. We end this subsection with the statement of a result of R. Taylor and V. Totik that gives a partial answer in the $d=1$ setting.

\begin{theorem}[\cite{TT}]\label{th:taylortotik} Let $K\subset \CC$ be compact and assume that the outer boundary of $K$ can be written as a finite union of $C^2$ arcs. Then any Leja sequence for $K$ satisfies property $(1)$.
\end{theorem}

In particular, Leja sequences on an interval satisfy property (1). We return to this topic in section 7.

\section{Recovering the function $V_K$.} In section 2.4 of \cite{BBCL}, an elementary argument showed that for arrays satisfying property (1) of Definition \ref{props}, the Lebesgue functions $\Lambda_n(z)$ can be used to recover $V_K$ in the sense that 
\begin{equation}\label{lebfcn}\lim_{n\to \infty} \log \Lambda_n(z)=V_K(z), \ z \in \CC^d.\end{equation}
This was proved in the univariate case but the same proof works in all dimensions. Property (2) is not sufficient for (\ref{lebfcn}) to hold. In this section, we investigate special families of polynomials which can be used to recover the extremal function $V_K$.

Zaharjuta's proof of the existence of the limit in (\ref{zahlim}) introduced the useful notion of {\it directional Chebyshev constants}. Let $e_1(z),...,e_j(z),...$ be a listing of the monomials
$\{e_i(z)=z^{\alpha(i)}=z_1^{\alpha_1}\cdots z_d^{\alpha_d}\}$ in
$\CC^d$ indexed using a lexicographic ordering on the multiindices $\alpha=\alpha(i)=(\alpha_1,...,\alpha_d)\in {\bf N}^d$, but with deg$e_i=|\alpha(i)|$ nondecreasing. Define the class of polynomials $$P_i=P(\alpha(i)):=\{e_i(z)+\sum_{j<i}c_je_j(z)\};  $$
and the Chebyshev constants
$$T(\alpha):= \inf \{||p||_K:p\in P_i\}.$$ 
We write $t_{\alpha,K}:=t_{\alpha(i),K}$ for a Chebyshev polynomial; i.e., $t_{\alpha,K} \in P_i$ and $||t_{\alpha,K}||_K=T(\alpha)$. Let
$\Sigma=\Sigma_d$ denote the standard simplex  in $\RR^d$; i.e.,
$$\Sigma = \{\theta =(\theta_1,...,\theta_d)\in \RR^d: \sum_{j=1}^d\theta_j=1, \ \theta_j\geq 0, \ j=1,...,d\},$$
and let
$$\Sigma^0 :=\{\theta \in \Sigma:  \ \theta_j > 0, \ j=1,...,d\}.$$
For  all $\theta \in \Sigma^0$, the limit
$$\tau(K,\theta):= \lim_{\alpha/|\alpha|\to \theta} T(\alpha)^{1/|\alpha|}$$
exists and is called the {\it directional Chebyshev constant} for $K$ in the direction $\theta$. Zaharjuta showed that
\begin{equation}\label{zahform}\delta(K)= \exp \bigl[\frac{1}{\hbox{meas}(\Sigma)}\int_{\Sigma^0} \log \tau(K,\theta)d\theta\bigr].\end{equation}

In \cite{bloompc}, the following theorem was proved.

\begin{theorem}[\cite{bloompc}]\label{togetvk} Let $K\subset \CC^d$ be compact, $L-$regular, and polynomially convex. Let $\{p_j\}$ be a sequence of polynomials such that for all $\theta \in \Sigma^0$, there exists a subsequence
$Y_{\theta}\subset {\bf Z}^+$ with $p_j\in P(\alpha_j), \ j\in Y_{\theta}$ and
$$\lim_{j\in Y_{\theta}} ||p_j||_K^{1/{\rm deg}p_j}=\tau(K,\theta).$$
Then
$$\bigl[\limsup_{j\to \infty}{1\over {\rm deg}p_j}\log {|p_j(z)|\over ||p_j||_K}\bigr]^*=V_{K}(z), \ z\not \in K. $$
\end{theorem}

The family $\{p_j\}$ is said to be $\theta-aT$ -- ``theta-asymptotically Chebyshev'' -- if the property in Theorem \ref{togetvk} holds. Bloom proved, in particular, that {\it Leja polynomials} associated to a Leja sequence have this $\theta-aT$ property. Using Theorem \ref{togetvk}, he proved an interesting corollary related to our condition (1). To this end, we begin with a {\it triangular} array $\{B_{sj}\}_{j=1,...,s; \ s=1,2,...}$ with the property that $VDM(B_{s1},...,B_{ss})\not = 0$ for each $s$. Define, for each multiindex $\alpha=\alpha(s)$, the polynomial 
$$G_{\alpha(s)}(z):=\frac{VDM(B_{s1},...,B_{ss},z)}{VDM(B_{s1},...,B_{ss})}.$$
Note that $G_{\alpha(s)}(z)=z^{\alpha(s)}+\sum_{j<s}c_j z^{\alpha(j)}\in P_s$. Moreover, it is straightforward to see that
$$G_{\alpha(s)}(z)=z^{\alpha(s)}-L_{\alpha(s)}(z^{\alpha(s)}) = t_{\alpha(s),K}-L_{\alpha(s)}(t_{\alpha(s),K})$$
where $L_{\alpha(s)}(f)$ is the LIP for $f$ and the points $B_{s1},...,B_{ss}$; i.e., 
$$L_{\alpha(s)}(f)(z)=\sum_{j=1}^s f(B_{sj})l_{sj}(z)$$
and
$$ l_{sj}(z)=\frac{VDM(B_{s1},...,B_{s(j-1)}, z, B_{s(j+1)}, ...B_{ss})}{VDM(B_{s1},...,B_{ss})}.$$
Letting 
$$\Lambda_{\alpha(s)}:=\sup_{z\in K}\sum_{j=1}^s |l_{sj}(z)|,$$
we have the following.

\begin{theorem}[\cite{bloompc}] \label{thm42} If $K$ is $L-$regular and $\lim_{|\alpha(s)|\to \infty}\Lambda_{\alpha(s)}^{1/|\alpha(s)|}=1$, then
$$V_K(z)=\bigl[\limsup_{|\alpha(s)|\to \infty} \frac{1}{|\alpha(s)|}\log \frac{G_{\alpha(s)}(z)}{||G_{\alpha(s)}||_K}\bigr]^*$$
for $z\in \CC^d \setminus \hat K$.
\end{theorem}

The hypothesis $\lim_{|\alpha(s)|\to \infty}\Lambda_{\alpha(s)}^{1/|\alpha(s)|}=1$ implies that the family of polynomials $\{G_{\alpha(s)}\}$ is $\theta-aT$. In particular, {\it Fekete polynomials} for each $s=1,2,...$ defined from an array that maximizes $|VDM(\zeta_1,...,\zeta_s)|$ over $(\zeta_1,...,\zeta_s)\in K^s$ are shown to have this property. Note that in this case, for $s=N=\dim \mathcal P_n$, the $s$ points $B_{s1},...,B_{ss}$ coincide with the $n-$th order (degree) Fekete points $A_{n1},...,A_{nN}$. A weighted version of Theorem \ref{togetvk} was proved as Theorem 3.5 of \cite{bloomlev2}. We will utilize this in section 6.

\section{Bernstein-Markov property.} For a compact set $K\subset \CC^d$ and a measure $\nu$ on $K$, we say that the pair $(K,\nu)$ satisfies a Bernstein-Markov property if there exist constants  $\{M_n\}$ with $\limsup_{n\to \infty} M_n^{1/n}=1$ and all polynomials $Q_n\in \mathcal P_n$ satisfy
$$ ||Q_n||_K\leq M_n||Q_n||_{L^2(\nu)}.$$
In \cite{BBCL} it was shown (Theorem 3.3) how one could recover the transfinite diameter $\delta(K)$ from asymptotics of Gram determinants associated to a Bernstein-Markov pair $(K,\nu)$. More recently, strong Bergman asymptotics were proved in \cite{BBN} in this setting: if $(K,\nu)$ satisfies a Bernstein-Markov property, then 
\begin{equation} \label{strongba} \frac{1}{N}B_n^{\nu} d\nu \to \mu_{K} \ \hbox{weak-}*
\end{equation}
where 
$$B_n^\nu(z):=\sum_{j=1}^N|q_j(z)|^2$$
is the $n-$th Bergman function for $K,\nu$ and 
$\{q_1,q_2,\cdots,q_N\}$ is an orthonormal basis for ${\mathcal P}_n$ with respect to $L^2(\nu)$. Thus it is natural to ask which compact sets $K$ admit measures $\nu$ satisfying a Bernstein-Markov property. The following result was proved in \cite{blldp}; since the proof is short, we include it.

\begin{proposition}[\cite{blldp}]  \label{strongbm} Let $K\subset \CC^d$ be an arbitrary compact set. Then there exists a probability measure $\nu$ such that $(K,\nu)$ satisfies a Bernstein-Markov property.
 \end{proposition}
 
 \begin{proof} To construct $\nu$, we first observe that if $K$ is a finite set, any measure $\nu$ which puts positive mass at each point of $K$ will work. If $K$ has infinitely many points, for each $k=1,2,...$ let $m_k =$dim$\mathcal P_k(K)$, the holomorphic polynomials on $\CC^{d}$ restricted to $K$. Then $\lim_{k\to \infty} m_k =\infty$ and $m_k\leq {d+k\choose k}=0((d)^k)$. For each $k$, let 
 $$\mu_k :=\frac{1}{m_k} \sum_{j=1}^{m_k} \delta_{z_j^{(k)}}$$
 where $\{z_j^{(k)}\}_{j=1,...,m_k}$ is a set of Fekete points of order $k$ for $K$; i.e., if $\{e_1,...,e_{m_k}\}$ is any basis for $\mathcal P_k(K)$, 
\begin{equation}\label{fekhere}\bigl |\det [e_i(z_j^{(k)})]_{i,j=1,...,m_k}\bigr|=\max_{q_1,...,q_{m_k}\in K}\bigl |\det[e_i(q_j)]_{i,j=1,...,m_k}\bigr|.\end{equation}
 Define
 $$\nu:=c\sum_{k=3}^{\infty} \frac{1}{k(\log k)^2}\mu_k$$
 where $c>0$ is chosen so that $\nu$ is a probability measure. 
 If $p\in \mathcal P_k(K)$, we have
 $$p(z)=\sum_{j=1}^{m_k} p(z_j^{(k)}) l_j^{(k)}(z)$$
 where $l_j^{(k)}\in \mathcal P_k(K)$ with $l_j^{(k)}(z_k^{(k)})=\delta_{jk}$. We have  
 $||l_j^{(k)}||_K=1$ from (\ref{fekhere}) and hence
 $$||p||_K \leq \sum_{j=1}^{m_k} |p(z_j^{(k)})|.$$
 On the other hand, 
 $$||p||_{L^2(d\nu)}\geq ||p||_{L^1(d\nu)}\geq  \frac{c}{k(\log k)^2}\int_K |p|d\mu_k$$
 $$= \frac{c}{km_k(\log k)^2}\sum_{j=1}^{m_k} |p(z_j^{(k)})|.$$
 Thus we have
 $$||p||_K \leq {km_k(\log k)^2\over c} ||p||_{L^2(d\nu)}.$$
\end{proof}

\noindent This gives a positive answer to the first part of question 2 in \cite{BBCL}. The second part has a negative answer, as the simple example of 
$$K=\{z\in \CC: |z|\leq 1\}\cup \{2\}$$
shows.

For certain measures $\nu$ with compact and non-polar support on the {\it real} line $\RR\subset \CC$, {\it pointwise} asymptotics of the Bergman functions $\{B_n^\nu\}$ are known (cf., \cite{totik}). In the higher dimensional setting, very little is known. Bos, et al \cite{BDM} consider one natural analogue of the interval, namely, the real unit ball 
$$B_d:=\{(z_1,...,z_d)\in \CC^d: \Im z_j  =0, \ j=1,...,d; \ \sum_{j=1}^d(\Re z_j)^2\leq 1\}$$
in $\RR^d \subset \CC^d$. Writing $x_j:=\Re z_j$ and $x=(x_1,...,x_d)$, it is known that 
$$d\mu_{B_d}=\omega^0(x)dx:=\frac{2}{\omega_d \sqrt {1-\sum_{j=1}^d x_j^2}}dx$$
where $dx=dx_1\wedge \cdots \wedge dx_d$ is $d-$dimensional Lebesgue measure and $\omega_d$ is the surface area of the unit sphere $S^d\subset \RR^{d+1}$. Lemma 1 in \cite{BDM} shows that if $d\nu(x)=\omega(x)dx$ where $\omega(x)=\omega(-x)$ is a positive {\it centrally symmetric} weight satisfying a certain Lipschitz property, then 
\begin{equation}\label{bosres}\lim_{n\to \infty} \frac{1}{N}B_n^{\mu_{B_d}}(x)= \frac{\omega^0(x)}{\omega(x)}\end{equation}
and the convergence is uniform on compact subsets of the $\RR^d-$interior of $B_d$.
The first step is the special case where $\omega(x)=\omega^0(x)$ (note in this case the right-hand-side of (\ref{bosres}) is the constant function $1$; it is also shown that the limit is the constant function $2$ on the sphere $|x|=1$).

Having a Bernstein-Markov measure allows one to replace Chebyshev polynomials by orthogonal polynomials in certain asymptotic computations. Using this idea, the exact calculation of the transfinite diameter of the real ball $B_d:=\{x\in \RR^d\,:\, |x|\le 1\}$ and the real unit simplex, $S =S_d:=\{x\in\RR^d_+\,:\,\sum_{j=1}^dx_j\le1\}$ in $\RR^d\subset \CC^d$ was recently achieved in \cite{boscalc}.

\begin{proposition}[\cite{boscalc}] \label{tdcalc} The transfinite diameter of the unit ball $B_d$ is:
\begin{enumerate}
\item for $d$ even,
\[\delta(B_d)=\frac{1}{2}\exp\left(-\frac{1}{4}\frac{2d+1}{d}\sum_{j=1}^d\frac{1}{j} +\frac{1}{2}+\frac{1}{2}\log(2) 
+\frac{1}{4d}\sum_{j=1}^d\frac{(-1)^j}{j}\right),
\]
\item for $d$ odd,
\[\delta(B_d)=\frac{1}{2}\exp\left(-\frac{1}{4}\frac{2d+1}{d}\sum_{j=1}^d\frac{1}{j} +\frac{1}{2}+\frac{d-1}{2d}\log(2) 
-\frac{1}{4d}\sum_{j=1}^d\frac{(-1)^j}{j}\right).
\]
\end{enumerate}
The transfinite diameter of the simplex $S_d$ is 
\[\delta(S_d)=(\delta(B_d))^2.\]

\end{proposition}

\section{Weighted vs. unweighted.} In the weighted theory, one considers {\it closed} sets which, for certain weights, may be unbounded. To be precise, let $K\subset \CC^d$ be closed and let $w$ be an admissible weight function on $K$:  $w$ is a nonnegative, usc function with
$\{z\in K:w(z)>0\}$ nonpluripolar; if $K$ is unbounded, we require that $w$ satisfies the growth property
\begin{equation} \label{grprop} |z|w(z)\to 0 \ \hbox{as} \ |z|\to \infty, \ z\in K. 
\end{equation}
Let $Q:= -\log w$ and define the weighted extremal function or weighted
pluricomplex Green function $V^*_{K,Q}(z):=\limsup_{\zeta \to z}V_{K,Q}(\zeta)$ where
$$V_{K,Q}(z):=\sup \{u(z):u\in L(\CC^d), \ u\leq Q \ \hbox{on} \ K\}. $$
In the unbounded case, property (\ref{grprop}) is equivalent to 
$$Q(z)-\log |z| \to \ +\infty \ \hbox{as} \ |z|\to \infty \ \hbox{through points in} \ K.$$ 
Due to this growth assumption for $Q$, $V_{K,Q}$ is well-defined and equals $V_{K\cap \mathcal B_R,Q}$ for $R>0$ sufficiently large where $\mathcal B_R=\{z:|z|\leq R\}$ (Definition 2.1 and Lemma 2.2 of Appendix B in \cite{safftotik}). It is known that the support
$$S_w:=\hbox{supp}(\mu_{K,Q})$$ 
of the {\it weighted extremal measure} 
$$\mu_{K,Q}:=\frac{1}{(2\pi)^d}(dd^cV_{K,Q}^*)^d$$
is compact; 
$$S_w \subset S_w^*:=\{z\in K: V_{K,Q}^*(z)\geq Q(z)\};$$
moreover, 
$$V_{K,Q}^*=Q \ \hbox{q.e. on} \ S_w$$
(i.e., $V_{K,Q}^*=Q$ on $S_w \setminus F$ where $F$ is pluripolar); and if $u\in L(\CC^d)$ satisfies $u\leq Q$ q.e. on $S_w$ then $u\leq V_{K,Q}^*$ on $\CC^d$. 

The unweighted case is when $K$ is compact and $w\equiv 1$ ($Q\equiv 0$); we then write $V_K:=V_{K,0}$ to be consistent with the previous notation. 

Even in one variable ($d=1$) the weighted theory introduces new phenomena from the unweighted case. As an elementary example, $\mu_K$ puts no mass on the interior of $K$ (in one variable, the support of $\mu_K$ is the outer boundary of $K$); but this is not necessarily true in the weighted setting. As a simple but illustrative example, taking $K$ to be the closed unit ball $\{z:|z|\leq 1\}$ and $Q(z)=|z|^2$, it is easy to see that $V_{K,Q}=Q$ on the ball $\{z:|z|\leq 1/\sqrt 2\}$ and $V_{K,Q}(z)=\log |z| + 1/2 - \log (1/\sqrt 2)$ outside this ball. One can check that if $K_R$ is the ball  $\{z:|z|\leq R\}$ for $1\leq R \leq \infty$ and $Q(z)=|z|^2$, one has $V_{K_R,Q}=V_{K,Q}$. 

Now let $K\subset \CC^d$ be compact and let
$w$ be an admissible weight function on
$K$.  Generalizing (\ref{vdmordn}), given $\zeta_1,...,\zeta_N\in K$, let
$$W(\zeta_1,...,\zeta_N):=VDM(\zeta_1,...,\zeta_N)w(\zeta_1)^{n}\cdots w(\zeta_N)^{n}$$
$$= \det
\left[
\begin{array}{ccccc}
 e_1(\zeta_1) &e_1(\zeta_2) &\ldots  &e_1(\zeta_N)\\
  \vdots  & \vdots & \ddots  & \vdots \\
e_N(\zeta_1) &e_N(\zeta_2) &\ldots  &e_N(\zeta_N)
\end{array}
\right]\cdot w(\zeta_1)^{n}\cdots w(\zeta_N)^{n}$$
be a (generalized) {\it weighted Vandermonde determinant} of order $n$. Let
$$W_n(K):=\max_{\zeta_1,...,\zeta_N\in K}|W(\zeta_1,...,\zeta_N)|$$
and define an {\it $n-$th order weighted Fekete set for $K$ and $w$} to be a set of $N$ points $\zeta_1,...,\zeta_ N\in K$ with the property that
$$|W(\zeta_1,...,\zeta_N)|=W_n(K).$$
We also write $\delta^{w,n}(K):=W_n(K)^{\frac{d+1}{dnN}}$ and define
$$ \delta^w(K):=\lim_{n\to \infty}\delta^{w,n}(K)=\lim_{n\to \infty}W_{n}(K)^{\frac{d+1}{dnN}}.$$
A proof of the existence of the limit may be found in \cite{BBnew} or \cite{[BL]}; in the latter work one 
defines the circled set 
	$$F=F(K,w):=\{(t,z)=(t,t \lambda )\in \CC^{d+1}: \lambda \in K, \ |t|=w(\lambda)\}$$
	and shows that, indeed, for the closure $\bar F$,
	$$\delta^w(K)=\delta (\bar F)^{\frac{d+1}{d}}.$$
	
	In \cite{BBN} the authors proved a weighted version of $(2) \implies (3)$:
	
	\begin{theorem}[\cite{BBN}] \label{asympwtdfek} Let $K\subset \CC^d$ be compact with admissible weight $w$. For each $n$, take points $x_1^{(n)},x_2^{(n)},\cdots,x_N^{(n)}\in K$ for which 
$$ \lim_{n\to \infty}|W(x_1^{(n)},\cdots,x_N^{(n)})|^{{(d+1)\over  dnN}}=\delta^w(K)$$
$(${\it asymptotically} weighted Fekete points$)$ and let $\mu_n:= \frac{1}{N}\sum_{j=1}^N \delta_{x_j^{(n)}}$. Then
$$
\mu_n \to \mu_{K,Q} \ \hbox{weak}-*.
$$
\end{theorem}

The main results in \cite{BBnew}, which are stated and proved in a much more general setting than weighted pluripotential theory in $\CC^d$ and lead to the results in \cite{BBN}, {\it require} the weighted theory. A self-contained exposition of the weighted pluripotential theoretic setting can be found in \cite{levsur}.	

As a final remark on weighted pluripotential theory, we provide a solution of Problem 3.4 in Appendix B of \cite{safftotik} in the locally regular, $w$ continuous case, correcting the end of section 8 in \cite{bloomlev2}. The problem is to show that the weighted Fekete polynomials can be used to recover the weighted pluricomplex Green function $V_{K,Q}$ in the sense of
Theorem 3.5 of \cite{bloomlev2}; see (\ref{amjfix}) below. Let
$K\subset \CC^d$ be locally regular and let
$w$ be a continuous admissible weight on $K$. We define {\it weighted Chebyshev constants} 
$$\tau_i^w(K):=\inf \{||w^{|\alpha (i)|}p||_K:p\in P_i\}^{1\over |\alpha (i)|}$$
and we let $t^w_{\alpha,K}$ denote a weighted Tchebyshev polynomial; i.e., $t^w_{\alpha,K}$ is of the form $w^{\alpha(i)}p$ with $p\in P_i$ and 
$||t^w_{\alpha,K}||_K= \tau_i^w(K)^{|\alpha(i)|}$.

Let $m_s={d+s \choose s}$. Note that $|\alpha(i)|=s$ for $m_{s-1}+1\leq i \leq m_s$. Given $\zeta_1,...,\zeta_ i\in K$, let
$$W(\zeta_1,...,\zeta_i):=VDM(\zeta_1,...,\zeta_i)w(\zeta_1)^{|\alpha (i)|}\cdots w(\zeta_i)^{|\alpha (i)|}.$$
Generalizing the notion of an $n-$th order weighted Fekete set for degree $n$, for each $i=1,2,...$, an $i-$th weighted Fekete set for $K$ and $w$ will be a set of $i$ points $\zeta_1,...,\zeta_ i\in K$ with the property that
$$W_i:=|W(\zeta_1,...,\zeta_i)|=\sup_{\xi_1,...,\xi_i\in K}|W(\xi_1,...,\xi_i)|.$$
Fix $i$ with $|\alpha(i)|=s$. We will define {\it weighted Fekete polynomials} $p_i$ for each positive integer $i$ with $|\alpha(i)|=|\alpha(i-1)|=s$; i.e., $m_{s-1}+1\leq i-1 \leq m_s$. Choose an $(i-1)-$st weighted Fekete set $\zeta_1,...,\zeta_ {i-1}$ and form the weighted polynomial
$$w(z)^{|\alpha(i)|}p_i(z)=w(z)^{|\alpha(i)|}\{e_i(z)+\sum_{j<i}c_je_j(z)\}:=\frac{W(\zeta_1,...,\zeta_{i-1},z)}{W_{i-1}}.$$
Thus 
$$\frac{W_i}{W_{i-1}}\geq ||w^{|\alpha(i)|}p_i||_K\geq \tau_i^w(K)^{|\alpha(i)|}.$$
Next we choose an $i-$th weighted Fekete set $\zeta_1,...,\zeta_ i$. In the expansion of the determinant of this weighted Vandermonde 
$$W(\zeta_1,...,\zeta_i):=VDM(\zeta_1,...,\zeta_i)w(\zeta_1)^{|\alpha (i)|}\cdots w(\zeta_i)^{|\alpha (i)|},$$
we replace the last row $w(\zeta_k)^{|\alpha(i)|}e_i(\zeta_k), \ k=1,...,i$ by $t^w_{\alpha(i),K}(\zeta_k), \ k=1,...,i$. Expanding the weighted Vandermonde determinant by the last row, 
$$W_i \leq \sum_{k=1}^i|t^w_{\alpha(i),K}(\zeta_k)|\cdot |W(\zeta_1,...,\zeta_{i-1}|$$
$$\leq i\tau_i^w(K)^{|\alpha(i)|}W_{i-1}.$$
Thus
$$\tau_i^w(K)^{|\alpha(i)|}\leq   ||w^{|\alpha(i)|}p_i||_K \leq  \frac{W_i}{W_{i-1}}\leq i\tau_i^w(K)^{|\alpha(i)|}.$$
The sequence of weighted polynomials $\{w^{|\alpha(i)|}p_i: \alpha(i)\not =(s,0,...,0), \ s=1,2,... \}$ satisfies the hypothesis of Theorem 3.5 of \cite{bloomlev2} -- note that $(1,0,...,0)\not \in \Sigma^0$ --  so that we obtain its conclusion:
\begin{equation}\label{amjfix}\bigl[\limsup_{i\to \infty}{1\over |\alpha(i)|}\log {|p_i(z)|\over ||w^{|\alpha(i)|}p_i||_K}\bigr]^*=V_{K,Q}(z), \ z\not \in \hat K.\end{equation}

\section{Explicit good interpolation points} 
An important question in numerical analysis and computational mathematics is to provide \textit{explicit} or \textit{computable} arrays $\{A_n\}=\{A_{nj}, \; j=1,\dots,N\}\subset K$ satisfying (4) in Definition \ref{props}. By this we mean points whose values can be entered on a computer with arbitrary precision to evaluate LIP's. Typical explicit points in the univariate setting are \textit{Chebyshev points} for which $$\{A_n\}=\left\{\cos \left(\frac{2k+1}{2(n+1)}\pi\right),\; k=0,\dots n\right\}, \quad  K=[-1,1],$$ or the \textit{roots of unity} for which $$\{A_n\}=\{\exp (2ik\pi/(n+1)),\; k=0,\dots,n\}, \quad K=\{|z|\leq 1\}.$$ Semi-explicit points, such as Fejer points (the image of a complete set of roots of unity under an exterior conformal mapping) or points related to orthogonal polynomials still deserve interest since, in particular cases and for $n$ not too large, they can be computed with high precision. The Fekete points for $[-1,1]$ are extreme points of Legendre polynomials plus $\{-1,1\}$. Recently Xu \cite{xu} constructed semi-explicit points on a region in $\RR^2$ bounded by two lines and a parabola using zeros of Jacobi polynomials. Typical non-explicit good points are multivariate Fekete points, which cannot currently be efficiently computed. Even in the one variable case, only a few examples of explicit good interpolation points are known. If $K$ is not an interval or a disk one generally uses \textit{algorithms} which provide numerical approximations for the points (cf., section 9). 

To clarify what we mean by good points for polynomial interpolation we utilize the following hierarchy which slightly refines some conditions listed in Definition \ref{props}. 
\begin{definition} Let $\{A_n\}$ be an array of interpolation points in $K\subset \CC^d$. We denote by $\Lambda_n$ the Lebesgue constant for $A_n,K$ and $L_nf$ the LIP of $f,A_n$. Here are four properties that $\{A_n\}$ may or may not possess: 
\begin{enumerate}
\item[(H1)] $L_nf\rightrightarrows f$ on $K$ for each $f$ holomorphic on a neighborhood of $K$, see Definition \ref{props}, (4). 
\item[(H2)] The Lebesgue constants $\Lambda_n$ grow sub-exponentially: $\lim_{n\to \infty} \Lambda_n^{1/n}$ $=1$, see Definition \ref{props}, (1).
\item[(H3)] The Lebesgue constants $\Lambda_n$ grow polynomially: there exists $s\in \NN$ such that $\Lambda_n=O(n^s)$ as $n\to\infty$. 
\item[(H4)] The Lebesgue constants $\Lambda_n$ grow sub-polynomially: $\Lambda_n=o(n)$ as $n\to\infty$.   
\end{enumerate} 
\end{definition}
Clearly $\textrm{(H4)} \implies \textrm{(H3)} \implies \textrm{(H2)} \implies \textrm{(H1)}$. 
None of the reverse implications is true. That there are sequences of points satisfying (H3) but not (H4) and (H2) but not (H3) will follow from the results below. In the multivariate case, since there is currently no analogue to Theorem \ref{walsh2}, we are left with the last three conditions and are obliged to study multivariate Lebesgue constants. In the univariate case, we may establish (H1) without having recourse to Lebesgue constants, e.g., in obtaining a discretization of the equilibrium measure (using (3) $\implies$ (4) in \eqref{eq:basicimplications}). An array of multivariate Fekete points is a fundamental example satisfying (H3).  

For theoretical approximation of holomorphic 
functions on a neighborhood of $K$, conditions (H4) and (H3) do not provide 
better results than (H2). However, computations with arrays of points having smaller
Lebesgue constants benefit from a higher stability and, from the point of view of approximation theory, one can derive convergence results for larger classes of functions. For instance, from  Lebesgue's inequality \eqref{lebineq} and Jackson's theorem, see e.g. \cite[\S 1.1.2]{rivlin}, if $K$ is a product of intervals in $\RR^d\subset \CC^d$ then (H3) implies that $L_nf\rightrightarrows f$ on $K$ for each $f$ which is $(s+1)$-times continuously differentiable on a neighborhood of $K$ while (H4) only requires $f$ to be continuously differentiable. In general, conditions (H3) and (H4) imply convergence results for classes of functions $f$ for which $d_n(f,K)$ (recall (\ref{dndef})) is known to decrease polynomially in $n$. Such estimates are known for several natural spaces of functions holomorphic on the interior of $K$ with some regularity up to the boundary (cf., \cite{shirokov} and the references therein).  

The array of Padua points $\{\mathbf{P}_n\}$ is an explicit example of multivariate interpolation points satisfying condition (H4). We follow the presentation given in \cite{padua1}. Another point of view can be found in \cite{padua2}. The points of $\mathbf{P}_n$  lie in $K=[-1,1]^2$ and are located on the classical Lissajous curve $t\in [0,\pi]\mapsto \gamma_n(t):=(\cos (nt), \cos((n+1)t))$. The Padua points are the double points of this curve together with its points on the (real) boundary of the square. A simple formula is the following:
\begin{equation} \mathbf{P}_n=\left\{\gamma_n\left(\frac{i}{n+1}\pi+\frac{j}{n}\pi\right)\;:\; i+j\leq n, \; (i,j)\in \NN^2\right\}. \end{equation}
It is readily seen that the above $\binom{n+2}{2}$ points are pairwise distinct but the fact that $\mathbf{P}_n$ is a unisolvent set of degree $n$ is not immediate. To see this, it suffices to exhibit a FLIP for any point $a\in \mathbf{P}_n$. There is a remarkable formula expressing such a FLIP with the help of the reproducing kernel $K_n$ for the inner product based on the tensor product of two arcsine measures, 
$$(p, q)\in \mathcal{P}_n \mapsto \langle p, q\rangle :=\frac{1}{\pi^2}\int_{[-1,1]^2} p(x,y)q(x,y)\frac{1}{\sqrt {1-x^2}}\frac{1}{\sqrt {1-y^2}}dx dy$$
$$=\frac{1}{\pi^2}\int_{[0,\pi]^2} p(\cos(t),\cos(s))q(\cos(t),\cos(s)) dtds.$$ 
The reproducing kernel $K_n$ is defined on $[-1,1]^2$ via the relation 
$$p(w)=\langle p, K_n(w\,;\,\cdot)\rangle, \quad p\in \mathcal{P}_n, \quad w=(x,y)\in [-1,1]^2.$$
It can be shown that the FLIP corresponding to $a=(x_a,y_a)\in \mathbf{P}_n$ is given by
$$l_a(x,y)=w_a \left\{K_n(a; (x,y))-T_n(y)T_n(y_a)\right\}$$
where $w_a$ is constant and $T_n$ is the ordinary Chebyshev polynomial of degree $n$.
The proof uses certain quadrature formulas for the the tensor product of two arcsine measures using the points of $\mathbf{P}_n$. Next, it can be shown that the kernel $K_n$, hence the FLIP's for $\mathbf{P}_n$, are expressible as a linear combination of quotients of classical trigonometric polynomials. A careful analysis leads to the following result. 
\begin{theorem}[\cite{padua1}] The Lebesgue constants $\Lambda_n$ for the Padua points $\mathbf{P}_n$ and $K=[-1,1]^2$ satisfy 
$$\Lambda_n=O(\log^2n), \quad n\rightarrow\infty.$$
In particular, $L_nf\rightrightarrows f$ on $K$ for each $f$ which is continuously differentiable on a neighborhood of $K$.
\end{theorem}
Lagrange interpolants at Padua points can be easily computed, cf., \cite{padua3, padua4}.   
Unfortunately, it is not clear what might be the analogues of Padua points in higher dimensions ($n>2$). 

We present another construction of good points based on a different idea. These points will only satisfy the weaker condition (H3) but there is a simple and efficient way of going from dimension $1$ to dimension $k$ for every $k$. The starting point is a classical algebraic formula giving multivariate interpolation points starting with univariate points. Given $d$ sets of $(n+1)$-tuples  $$ \mathbf{A_s}=\left(a_{0s}, \dots, a_{ns}\right),\quad s=1,\dots,d$$ consisting of distinct points in $\CC$, we intertwine these tuples: 
$$ \mathbf{A_1}\oplus\mathbf{A_2}\oplus \cdots \oplus \mathbf{A_d} :=\left\{\left(a_{i_11},a_{i_22},\cdots,a_{i_dd}\right),\;  i_1+i_2+\cdots+i_d \leq d\right\}.$$
Changing the ordering of the points would provide a different set of points in $\CC^d$. The set of points we obtain is unisolvent of degree $n$ in $\CC^d$. Such points were studied by Siciak \cite{siciak2} and used in \cite[Theorem 4.8]{BBCL} to prove the existence of points satisfying properties (2), (3) and (4) in Definition \ref{props}. We want to construct interpolation points by  intertwining well-chosen -- and well-ordered -- univariate interpolation points which satisfy (H3). The obvious strategy is to try to relate the Lebesgue constants of the multivariate interpolations points to the Lebesgue constants of the univariate points. Such a relation is given in the following theorem. For simplicity, we state only the case $d=2$.

\begin{theorem}  [\cite{calvi}] \label{whatthef} Let $K$ be a compact set in $\CC^{2}$ containing $\mathbf{A_1}\oplus \mathbf{A_2}$. We let $K_1$ $(K_2)$ denote the projection of $K$ on the $z_1$ $(z_2)$ axis. Then
$$\Lambda_n(\mathbf{A_1}\oplus \mathbf{A_2} | K)\leq 4\binom{n+2}{n}\sum_{i+j\leq n}
\Lambda_n\left(\mathbf{A_1}^{[i]} | K_1\right) \cdot\Lambda_n\left(\mathbf{A_2}^{[j]} | K_2\right)
$$ where $\mathbf{A_1}^{[i]}=(a_0, \dots, a_i)$ and $\Lambda_n\left(\mathbf{A_1}^{[i]} | K_1\right)$ denotes its Lebesgue constant with respect to the compact set $K_1$ $($likewise for $\mathbf{A_2}^{[j]})$. 
\end{theorem}
In order to estimate the Lebesgue constant $\Lambda_n(\mathbf{A_1}\oplus \mathbf{A_2} | K)$, bounds on $\Lambda_n(\mathbf{A_1}|K_1)$ and on $\Lambda_n(\mathbf{A_2}|K_2)$ do not suffice. We must have bounds on the Lebesgue constants of {\it every} subset $\mathbf{A_1}^{[i]}$ and $\mathbf{A_2}^{[j]}$ for $i+j\leq n$. The only practical way of using the theorem is to start with univariate  points given by a \textit{sequence} of interpolation points $\mathbf{A_s^n}=(a_{0s},\dots, a_{ns})$, so that for every $i\leq n$, $\mathbf{A_s^n}^{[i]}=A_s^i$. Then the search for good multivariate interpolation points via the intertwining process is reduced to the problem of finding univariate interpolation points given by a sequence and satisfying (H3). Surprisingly, such sequences did not seem to be known until recently. \par 
We now discuss the construction of such univariate sequences. All examples currently available are constructed with the help of Leja sequences for the closed unit 
disk. Recall that a Leja sequence for a compact set $K\subset \CC$ is a sequence $\{a_n\}$ in $K$ such that 
$$\max_{z\in K} \prod_{i=0}^d |z-a_i| = \prod_{i=0}^d |a_{d+1}-a_i|, \quad d\geq 0.$$
If we are to produce {\it explicit} points we must restrict to $K=D=\{|z|\leq1\}$. 
In this case, the structure of Leja sequences is given by the following result.
We always assume that the first term equals $1$. 
\begin{theorem}[\cite{biacal}]\label{th:biacal} Leja sequences for the unit disk $D$ satisfy:
\begin{enumerate}
\item A $2^n$-Leja section is formed by the $2^n$-th roots of unity.
\item If $E_{2^{n+1}}$ is a $2^{n+1}$-Leja section then there exist a $2^n$-th root $\rho$ of $-1$ and a $2^n$-Leja section $U_{2^n}$ such that $E_{2^{n+1}}=(E_{2^n}\, ,\, \rho U_{2^n}).$
\end{enumerate}
\end{theorem}
Using this result, the following estimates were recently established. 
\begin{theorem}[\cite{cp1}] Let $\{e_j\}$ be a Leja sequence for $D$. As $n\rightarrow\infty$, $\Lambda_n=O(n\log n)$ where $\Lambda_n$ is the Lebesgue constant for $\{e_0,\dots,e_{n-1}\}$. 
\end{theorem}

A similar estimate holds for the image of Leja sequences under external conformal mappings $\overline{\CC}\setminus D\rightarrow \overline{\CC}\setminus K$ for $K$ sufficiently regular, e.g., $K$ bounded by a $C^2$ Jordan curve \cite{cp1}.  It is also shown in \cite{cp1} that a Leja sequence for $D$ cannot satisfy (H4), thus showing that, in general, (H3) does not imply (H4). 

For practical applications, real points are more useful. A simple idea to construct such points is to project a Leja sequence for $D$ onto the real axis. Since a Leja sequence for $D$ is symmetric with respect the real axis, complex conjugate points provide the same real point. Eliminating this redundancy we obtain a so-called $\Re$-Leja sequence \cite{cp2}. One can specify the $n$-th entry of a $\Re$-Leja sequence in terms of the real part of a certain entry of the Leja sequence used in its construction. 
 
\begin{theorem}[\cite{cp2}]
Let $X=\{x_j\}$ be a $\Re$-Leja sequence. The Lebesgue constants $\Lambda_n$ for the points $x_0, \dots, x_{n-1}$ satisfy
\begin{equation} \Lambda_n=O(n^3\log n), \quad n\to\infty. \end{equation} \end{theorem}

By combining the above two results and the general version of Theorem \ref{whatthef} we obtain the following result. 

\begin{theorem}[\cite{cp1,cp2}] Intertwining $k$ Leja sequences for $D$ and $d-k$ $\Re$-Leja sequences for $[-1,1]$ yields unisolvent sets on the Cartesian product of these sets in $\CC^d$ whose Lebesgue constants $\Lambda_n$ satisfy $(\rm{H}3)$. \end{theorem}

Moreover, the degree of the polynomial growth of $\Lambda_n$ can be estimated. One can give an explicit expression for the $n$-th element of certain simple Leja (or $\Re$-Leja) sequences that depend  only on the binary expansion of the index $n$. Details can be found in \cite{cp1,cp2}. In view of Theorems \ref{whatthef} and \ref{th:taylortotik}, the intertwining of Leja sequences for many compact sets provides further examples of multivariate interpolation points satisfying property (H2). Goncharov \cite{goncharov} has constructed a sequence in $[-1,1]$ satisfying (H2) but not (H3) by arranging the classical Chebyshev points in a certain manner.

\section{The quest for good points in the real disk}
We describe a natural strategy for finding good points in the real disk $B=B_2=\{(x,y)\in \RR^2: x^2+y^2\leq 1\}$. Using Theorem 3.3, one can calculate the transfinite diameter of $B$ to find
$$\delta(B)=\frac{1}{\sqrt {2  e}};$$
and its equilibrium measure is well-known, 
$$d\mu_B = \frac{r}{2\pi\sqrt{1-r^2}}drd\theta$$
in polar coordinates in $\RR^2$; i.e., $d\mu_B$ is absolutely continuous with respect to Lebesgue measure on $B$ with density $\frac{1}{2\pi\sqrt{1-r^2}}$.  

It was shown in \cite{BBCL} that $(3)\nRightarrow (1),(2),(4)$. The set $B$ was used in construction of {\it Bos arrays} satisfying (3) but not (4) (and (3) but not (2)). The points at the $n-$th stage in a Bos array are formed by taking a union  of equally spaced points on concentric circles centered at the origin. Precisely, if $n=2s$ is even, one chooses $s+1$ radii $R_{s0} < \cdots < R_{ss}=1$ and $4j+1$ equally spaced points on the circle of radius $R_{sj}$. The Vandermonde determinant $|VDM(A_{n1},...,A_{nN})|$ depends only on the radii $R_{s0}, \cdots , R_{ss}$ and if the asymptotic distribution of the radii on $[0,1]$ is given by a function $G:[0,1]\to [0,1]$; i.e., if $G(\frac{j}{s+1})=R_{sj}^2$, then 
$$\lim_{n\to \infty} |VDM(A_{n1},...,A_{nN})|^{1/l_n} = \frac{1}{\sqrt 2}\exp {\frac{3}{4L(G)}}$$
where 
$$L(G)=\int_0^1x^2\log G(x) dx+ 2\int_0^1\int_x^1 x\log [G(y)-G(x)]dydx.$$
Thus, if one could construct a Bos array with $L(G)=-2/3$, then this array would satisfy (2). 

Taking $G(x)=(1-\cos \pi x)/2$, the radii distribute asymptotically like the Chebyshev distribution on $[0,1]$ and this is a necessary condition (see \cite{BBCL}) that such an array satisfies (4). We state without proof an interesting calculation.

\begin{lemma}
For $G(x)=(1-\cos(\pi x))/2=\sin^2(\pi x/2),$
\[L(G)=-\frac{4}{3}\log(2)+\frac{2}{\pi^2}\zeta(3)\approx-0.6806085842\cdots\]
where $\zeta(x)$ is the classical zeta function.
\end{lemma}

\noindent In particular with this $G$, $L(G)\not = -2/3$ so such a Bos array does not satisfy (2). 

Taking $G(x)=1-(x^2-1)^2$, the arrays satisfy (3): we obtain $\mu_B$ as the limiting measure. Elementary but nontrivial calculations yield

\begin{lemma} 
For $G(x)=1-(x^2-1)^2,$
\[L(G)=-\frac{26}{9}-4\log(2)+4\sqrt{2}\log(\sqrt{2}+1)\approx
-0.675675691\cdots.\]
\end{lemma}

Thus again, with this $G$, $L(G)\not = -2/3$ so such a Bos array does not satisfy (2). Indeed, if one {\it could} find a Bos array with $L(G)=-2/3$, then, by Theorem \ref{asympfek}, the array satisfies (3) and hence, a posteriori, $G(x)=1-(x^2-1)^2$, a contradiction. Thus, unfortunately, Bos arrays on $B$ {\it never} satisfy (2), giving a negative answer to question 9 of \cite{BBCL}.

\section{Algorithms}
It is clear that one can expect to have lists of explicit good interpolation points
only for a very limited class of compact sets, even in the univariate case. If we have
to produce good points for a more or less arbitrary compact set, one must produce them 
algorithmically. We now discuss some recent work in this direction. 
\par
In a series of papers, Bos, Sommariva and Vianello (cf. \cite{SV} and
with De Marchi \cite{BDSV}) have introduced the notion of {\it approximate
Fekete points}. For $K\subset\CC^d$ compact, a basis $\{P_1,P_2,\cdots,P_N\}$ 
for ${\mathcal P}_n$, and a set of $M\geq N$ points $\{a_1,...,a_M\}$ of $K$, we consider the $N\times M$ matrix whose columns are of the form
$${\vec V}(a_j):=
\left[
\begin{array}{ccccc}
 P_1(a_j)\\
  P_2(a_j)\\ 
  \vdots  \\
  P_N(a_j)
\end{array}\right].$$
Selecting 
a subset of columns is then equivalent to selecting a subset of points. We choose the first point $x_1\in \{a_1,...,a_M\}$ to maximize $\|{\vec V}(a_j)\|_2$. Having chosen  $x_1,x_2,\cdots,x_k\in  \{a_1,...,a_M\}$ the $(k+1)$st point $x_{k+1}\in \{a_1,...,a_M\}$ is chosen
so that the volume generated by the columns ${\vec V}(x_{k+1})$ and
${\vec V}(x_1),{\vec V}(x_2),\cdots,{\vec V}(x_{k})$ is as large as possible. 

Suppose $K$ is $L-$regular. If for each $n=1,2,...$ one chooses a set of $M(n)\geq N$ points $A_{M(n)}=\{a_1^{(n)},...,a_{M(n)}^{(n)}\}$ of $K$ so that $\bigcup_n A_{M(n)}$ forms a {\it weakly admissible mesh for $K$} (WAM), then the corresponding array of approximate Fekete points satisfies (2) and hence (3) (Theorem 1 of \cite{BCLSV}). The mesh $\bigcup_n A_{M(n)}$ is weakly admissible, according to \cite{calvilev}, if $\# A_{M(n)}$ grows polynomially in $n$ and 
$$ ||p||_K\leq C_n ||p||_{A_{M(n)}} \ \hbox{for all} \ p\in \mathcal P_n$$
where $C_n$ grows polynomially in $n$. All $L-$regular compact sets $K$ admit a weakly  admissible mesh; cf., Theorem 16 of \cite{calvilev}. We remark that a WAM is called {\it admissible} (AM) if one can take $C_n=C$, a constant independent of $n$. 

There is also an algorithmic notion of {\it discrete Leja points}; as with approximate Fekete points, constructing discrete Leja points from a weakly admissible mesh $\bigcup_n A_{M(n)}$ gives an array satisfying (2) and hence (3). The interested reader is referred to \cite{BDSV} for details of the algorithm.

\section{Kergin interpolation} Of the many polynomial interpolation alternatives to Lagrange interpolation, one of the most productive ones utilized for interpolating holomorphic functions in $\CC^d, \ d>1$ is {\it Kergin interpolation}. \par In this section, we give a new presentation of Kergin interpolation which highlights its canonical character. 
We let $\mathcal O(\CC^d)$ denote the space of entire functions and $\mathcal{L}(\mathcal O(\CC^d), \mathcal{P}_n)$ the space of continuous linear maps from $\mathcal O(\CC^d)$ to $\mathcal{P}_n$. 

\begin{theorem}\label{th:TC} There exists a unique map $\mathcal{K}$,
\begin{equation} \mathcal{K}\,:\, \mathbf{A}=(a_0,\dots ,a_n)\in(\CC^d)^{n+1} \longrightarrow \mathcal{K}[\mathbf{A}]\in \mathcal{L}(\mathcal O(\CC^d), \mathcal{P}_n),\end{equation}
such that
\begin{enumerate}
	\item[($K1$)]  for every $f\in \mathcal O(\CC^d)$, $\mathcal{K}[\mathbf{A}](f)(a_j)=f(a_j), \ j=0,...,n$;
	\item[($K2$)] for every $f\in \mathcal O(\CC^d)$, the map $\mathbf{A}\rightarrow \mathcal{K}[\mathbf{A}](f)$ is continuous;
	\item[($K3$)] $\mathcal{K}$ is coordinate-free.
\end{enumerate}
The map $\mathcal K$ is defined by
\begin{equation}\label{eq:cfki}
\mathcal{K}[\mathbf{A}](f)(x)=\sum_{k=0}^{n} \int_{S_k} D^kf \left(\sum_{j=0}^k t_ia_i\right)(x-a_0,\dots, x-a_{k-1}) dm_k(t), 
 \end{equation}
where $D^kf$ is the $k$-th total derivative of $f$,  $S_k=\{ t=(t_0,\dots,t_k)\in [0,1]^{k+1}\; :\; \sum_{i=0}^k t_i=1\}$ and $dm_k$ is Lebesgue measure on $S_k$. 
\end{theorem}
That $\mathcal{K}[\mathbf{A}]$ is coordinate-free means that for every invertible linear map $m$ on $\CC^d$ 
$$ \mathcal{K}[\mathbf{A}] (\cdot \circ  m ) =\mathcal{K}[m\mathbf{A}](\cdot) \circ m, $$
where $m\mathbf{A}:=(m(a_0),\dots , m(a_d))$. \par
Condition ($K1$) is relatively weak (e.g., if $a_i=a$ for $i=0,\dots,d$ there is only one condition). We shall see later that, together with ($K2$) and ($K3$), it implies much stronger properties.  The operator ${\mathcal K} [\mathbf{A}]$ is called the {\it Kergin interpolating operator} with respect to ${\mathbf A}$. In contrast with multivariate Lagrange interpolation, the number of points $n+1$ is independent of the dimension $d$ of $\CC^d$.

\begin{proof} We first prove that there exists at most one map $$\Pi\,:\, \mathbf{A}\in(\CC^d)^{n+1} \rightarrow \Pi[\mathbf{A}]\in \mathcal{L} (\mathcal O(\CC^d), \mathcal{P}_n)$$ satisfying ($K1$), ($K2$) and ($K3$). 
Suppose $\Pi_1$ and $\Pi_2$ are two such maps. We prove that for every $\mathbf{A}\in U:=\{\mathbf{A}\in (\CC^d)^{n+1} \,:\, \textrm{$a_l\not=a_j$ for $l\not=j$}\}$ and every $f\in \mathcal O(\CC^d)$, $\Pi_1[\mathbf{A}](f)=\Pi_2[\mathbf{A}](f)$. Since
 $U$ is dense in $(\CC^d)^{n+1}$ and $\mathbf{A}\rightarrow  \Pi_i[\mathbf{A}](f)$ is continuous, this suffices to prove our claim. \par We can reduce 
the problem as follows. Since $\Pi_i[\mathbf{A}]$ is continuous on $\mathcal O(\CC^d)$ and the space $V$ spanned  by ridge entire functions --  functions of the form $z\mapsto h(\langle \lambda, z\rangle)$, where $h\in \mathcal O(\CC)$ and $\lambda\in \CC^d$ -- is dense in $\mathcal O(\CC^d)$, it suffices to prove that  $\Pi_1[\mathbf{A}]=\Pi_2[\mathbf{A}]$ on $V$.  Further, since $\Pi_i[\mathbf{A}]$ is linear we simply need to prove $\Pi_1[\mathbf{A}](f)=\Pi_2[\mathbf{A}](f)$ for $f=h(\langle \lambda , \cdot\rangle)$, with $h\in \mathcal O(\CC)$ and $\lambda\in \CC^d$, $\lambda\not=0$. 
\par
Fixing such an $f$, let $H= \{\langle\lambda, \cdot\rangle=0\}$ be the hyperplane orthogonal to $\lambda$. For $\epsilon > 0$ we define a linear map $m_\epsilon$ by $m_\epsilon (\lambda)=\lambda$ and ${m_\epsilon}|_{H}=\epsilon Id$ where $Id$ denotes the identity on $H$. Clearly $m_\epsilon$ is invertible.
 Moreover, since $m_\epsilon x-x\in H $, we have 
$$(f\circ m_\epsilon) (x)= h(\langle \lambda , m_\epsilon x \rangle)=h(\langle \lambda ,  x \rangle)=f(x).$$
Since $\Pi_i$ is coordinate-free, we deduce that
$$\Pi_i[\mathbf{A}](f)= \Pi_i[\mathbf{A}](f\circ m_\epsilon)=\Pi_i[m_\epsilon\mathbf{A}](f)\circ m_\epsilon, \quad \epsilon >0.$$ 
We have $m_\epsilon (x)\rightarrow \langle \lambda, x \rangle \frac{\lambda}{\|\lambda\|^2}$ as $\epsilon \rightarrow 0$, thus by ($K2$),
\begin{equation}\label{eq:CT2}\Pi_i[\mathbf{A}](f)(x)= \Pi_i\left[\langle \lambda \mathbf{A}\rangle \cdot\frac{\lambda}{\|\lambda\|^2} \right](f)\left( \langle \lambda, x \rangle\cdot \frac{\lambda}{\|\lambda\|^2}\right)\end{equation}
where $\langle \lambda \mathbf{A}\rangle \cdot\frac{\lambda}{\|\lambda\|^2}=(\langle \lambda a_i\rangle \cdot\frac{\lambda}{\|\lambda\|^2}\, :\, i=0,\dots,n)$. Since $\Pi_i$ takes values in $\mathcal{P}_n$, (\ref{eq:CT2}) implies that there exists a univariate polynomial $p$ of degree at most $n$ depending on $f$, $A$ and $\lambda$ such that
$$\Pi_i[\mathbf{A}](f)(x)=p(\langle \lambda , x\rangle).$$
We specialize to the case where the $\langle \lambda , a_i\rangle$ are distinct. Since the $a_i$ themselves are distinct, the set $\tilde U$ of all such $\lambda$ is dense in $\CC^d$. It remains to use 
assumption ($K1$). We have
$$h(\langle \lambda , a_i\rangle)=f(a_i)=\Pi_i[\mathbf{A}](f)(a_i)=p(\langle \lambda , a_i\rangle), \quad i=0,\dots,n.$$
Hence $p$ is a polynomial of degree at most $n$ that interpolates $h$ at these $n+1$ points, i.e, $p_n$ is the LIP of $h$ at these points which we write as 
\begin{equation}\label{eq:ALI}\Pi_i[\mathbf{A}](f)=L[\langle \lambda , a_0\rangle , \dots \langle \lambda , a_n\rangle\, ;\, h](\langle \lambda , \cdot \rangle), \quad \lambda\in \tilde U.\end{equation}
In particular, $\Pi_1[\mathbf{A}](f)=\Pi_2[\mathbf{A}](f)$. We now use the density of $\tilde U$ and the continuity of $f\rightarrow \Pi_i[\mathbf{A}](f)$ to extend the identity to the case where $\lambda\not\in \tilde U$. This finishes the proof of the uniqueness. 

\medskip

Identity (\ref{eq:ALI}) shows that if a map $\Pi$ with the required properties 
exists then it should come as a natural multivariate generalization of one 
of the many available expressions of univariate Lagrange-Hermite interpolation. Formula \eqref{eq:cfki} is the natural multivariate version of the classical Hermite-Genocchi formula. The proof that this map satisfies the required properties is a simple calculation; cf., \cite{micchelli}.
\end{proof}

It is not difficult to show that the map ${\mathcal K} [\mathbf{A}]$ interpolates in the Hermite sense; i.e., if a point $a$ appears $k$ times in $\mathbf{A}$ then $D^j {\mathcal K} [\mathbf{A}] (f)(a)=D^jf(a)$, $j=0,\dots,k-1$. Kergin interpolating operators enjoy many interesting algebraic properties including the following.
\begin{enumerate} 
\item ${\mathcal K}[{\mathbf A}]$ is independent of the ordering of the points in ${\mathbf A}$, and
\item ${\mathcal K}[\mathbf{B}]\circ {\mathcal K}[{\mathbf A}]={\mathcal K}[{\mathbf B}]$ for every $\mathbf B\subset {\mathbf A}$.
\end{enumerate}
 
In Theorem \ref{th:TC}, Kergin operators are defined only for entire functions. Andersson and Passare \cite{AP, passare}, showed that Kergin operators ${\mathcal K}_D$ can actually be defined on $\mathcal O(D)$ where $D$ is a $\CC$-convex domain in $\CC^d$, i.e., the intersection of $D$ with any complex line is connected and simply connected. In $\RR^d$ this is simply ordinary convexity if we replace ``complex line'' by ``real line.'' 

There are many results on the approximation of holomorphic functions by Kergin polynomials. We offer a brief sample. \par 
Let $K\subset D$ be compact and set ${\mathcal K}_n:={\mathcal K}_D[{\mathbf A}_n]$ where,  for $n=1,2,3,\ldots$, ${\mathbf A}_n=[A_{n0},\ldots,A_{nn}]\subset K$. For $D$ with $C^2-$boundary, Bloom and Calvi \cite{BC1} gave conditions on the array $\{{\mathbf A}_n\}_{n=1,2,\ldots}$ so that ${\mathcal K}_n(f)$ converges to $f$ uniformly on $K$ as $n\to {\infty}$ for every function $f$ holomorphic in some neighborhood of $\bar D$. They utilized an integral representation formula for the remainder $f-{\mathcal K}_n(f)$ proved by Andersson and Passare \cite{AP}. 

More in line with the ideas in this work, we call an array $\{{\mathbf A}_n\}_{n=1,2,\ldots}$ {\it extremal} for a compact set $K$ if ${\mathcal K}_n(f)$ converges to $f$ uniformly on $K$ for each  $f$ holomorphic in a neighborhood of $K$. For $K\subset \RR^d$, Bloom and Calvi \cite{BC2} proved the following striking result.

\begin{theorem} Let $K\subset \RR^d, \ d\geq 2$, be a compact, convex set with nonempty interior. Then $K$ admits extremal arrays for Kergin interpolation if and only if $d=2$ and $K$ is the region bounded by an ellipse.
\end{theorem}

\noindent Thus, for example, the real disk 
$$B=B_2=\{(z_1,z_2)\in \CC^2: \Im z_1  = \Im z_2=0, \ (\Re z_1)^2+ (\Re z_2)^2\leq 1\}$$
admits extremal arrays for Kergin interpolation.

\section{Open problems.} 
We conclude with some open questions, a subset of which comes from \cite{BBCL}.

\begin{enumerate}
\item Is the converse of Proposition \ref{onetwo} true?
\item Does $(2)\implies (4)$?
\item If an array lies in the Shilov boundary $S_K$ of $K$, does $(4)\implies (3)$? The Chebysev-radii Bos array in $B_2$ described in section 8 might give a counterexample. 
\item Construct an explicit array in the ball $B_2$ in section 8 satisfying (2), or, even better, (1).
\item Find an example of a compact set $K\subset \CC^d, \ d>1$, for which one can explicitly construct Fekete points.
\item Do multivariate Leja sequences satisfy (1)? (4)?
\item One can define multivariate {\it weighted} Leja sequences; starting with any point $x_1\in K$, having chosen $x_1,...,x_m\in K$ we choose $x_{m+1}\in K$ so that
$$|W(x_1,...,x_m,x_{m+1})|=\max_{x\in K} |W(x_1,...,x_m,x)|.$$
 Do these yield asymptotic weighted Fekete arrays?
 \item For $K \subset \CC^d$ compact and $L-$regular, does there exist $c=c(K)>1$ such that Fekete arrays of order $cn, \ n=1,2,É$, 
form an admissible mesh (AM) for $K$?
\item For $K \subset \CC^d$ compact, $L-$regular, and polynomially convex, if a {\it triangular} array satisfies $\{G_{\alpha}\}$ is $\theta- aT$ for $K$, is (4) satisfied? Is the converse true? Note if $d=1$ this equivalence is (essentially) Theorem \ref{walsh2}.
\item Let $K\subset \CC^d, \ d>1,$ be $L-$regular. If one takes {\it asymptotic} Fekete points, can the corresponding polynomials be used to recover the pluricomplex Green function $V_{K}$ as in Theorem \ref{thm42}?

\end{enumerate}

\end{document}